\documentclass{IEEEtran4PSCC}
\ifCLASSINFOpdf
   \usepackage[pdftex]{graphicx}
\else
   \usepackage[dvips]{graphicx}
\fi
\usepackage[cmex10]{amsmath}
\usepackage{mathtools}
\usepackage{amsthm}
\usepackage{amsfonts}
\usepackage{amssymb}
\usepackage{xcolor}

\usepackage{algorithm} 
\usepackage{algpseudocode} 

\usepackage{booktabs}
\usepackage{array}

\newtheorem{theorem}{Theorem}[section]
\newtheorem{lemma}[theorem]{Lemma}
\newtheorem{proposition}[theorem]{Proposition}
\newtheorem{corollary}[theorem]{Corollary}

\theoremstyle{definition}
\newtheorem{definition}[theorem]{Definition}

\theoremstyle{remark}
\newtheorem{remark}[theorem]{Remark}

\newcommand{\R}{\mathbb{R}}

\newcommand{\C}{\mathbb{C}}

\newcommand{\Smat}{\mathbb{S}}
\newcommand{\Hmat}{\mathbb{H}}
\newcommand{\trace}{\text{trace}}

\makeatletter
\let\old@ps@headings\ps@headings
\let\old@ps@IEEEtitlepagestyle\ps@IEEEtitlepagestyle
\def\psccfooter#1{%
    \def\ps@headings{%
        \old@ps@headings%
        \def\@oddfoot{\strut\hfill#1\hfill\strut}%
        \def\@evenfoot{\strut\hfill#1\hfill\strut}%
    }%
    \def\ps@IEEEtitlepagestyle{%
        \old@ps@IEEEtitlepagestyle%
        \def\@oddfoot{\strut\hfill#1\hfill\strut}%
        \def\@evenfoot{\strut\hfill#1\hfill\strut}%
    }%
    \ps@headings%
}
\makeatother

\psccfooter{%
        \parbox{\textwidth}{\hrulefill \\ \small{24th Power Systems Computation Conference} \hfill \begin{minipage}{0.2\textwidth}\centering \vspace*{4pt} \includegraphics[scale=0.06]{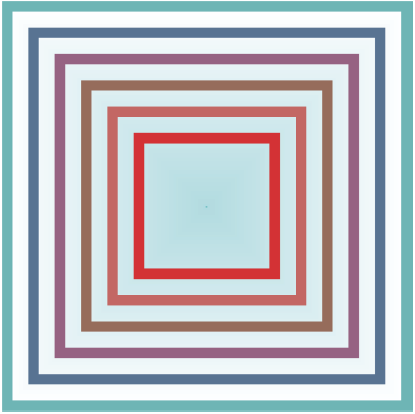}\\\small{PSCC 2026} \end{minipage} \hfill \small{Limassol, Cyprus --- June 8-12, 2026}}%
}

\begin{document}
%
\title{Advanced Cutting-Plane Algorithms for ACOPF}

\author{
\IEEEauthorblockN{Daniel Bienstock\\}
\IEEEauthorblockA{IEOR, Columbia University\\
New York City, USA\\
dano@columbia.edu}
\and
\IEEEauthorblockN{Mat\'ias Villagra\\}
\IEEEauthorblockA{ IEOR, Columbia University\\
New York City, USA \\
mjv2153@columbia.edu}
}


\maketitle

\begin{abstract}
We propose a disciplined, numerically stable, and scalable approach to SDP relaxations of the ACOPF problem based on linear cutting-planes. Our method can be warm-started and, owing to its linear nature, enables the computation of tight and accurate bounds for large-scale multi-period relaxations -- well beyond what nonlinear convex solvers can achieve. Preliminary experiments show promising results when benchmarked against state-of-the-art bounds on PGLIB instances.
\end{abstract}



\section{Introduction}

In this paper we propose a new paradigm to address semidefinite programming (SDP) relaxations for the Alternating-Current Optimal Power Flow Problem (ACOPF). We leverage SDP tightness by using linear outer-approximations without solving a single SDP, while taking advantage of mature linear programming solver technology. As in~\cite{bienstock+villagra25}, our procedure can naturally operate in \emph{warm-start} mode, which opens the door for efficient reoptimization of sequential and multi-period ACOPF solves. We believe that, given current optimization solver capabilities, this is the right direction for numerically robust, reliable computation of tight lower bounds on massive multi-period or sequential instances from the largest public test sets.

To this end, we first require a rigorous understanding of the different SDP formulations. The literature on convex relaxations for ACOPF is extensive, and it is understood that the SDP relaxations are frequently the tightest. On the other hand, it seems to be folklore knowledge that there are no differences between the real~\cite{bai+etal08,lavaei+low12} and complex~\cite{bose+etal11} SDP relaxations -- it is widely believed that there exists only \emph{one} SDP relaxation. This commonly held and yet, to the best of our knowledge, unproven belief -- please refer to Section~\ref{section:SDP-C_SDP-R} for a discussion -- has played a central role in the literature on SDP relaxations, and, more broadly, on convex relaxations of the ACOPF problem. Our second goal is to address this apparent lack of clarity by proving that the real SDP relaxation is at least as strong as the complex SDP.

In the current state-of-the-art for ACOPF, some interior point methods are empirically successful at computing very good solutions but cannot provide any bounds on solution quality. In our recent paper~\cite{bienstock+villagra25} we develop scalable and numerically stable linearly-constrained relaxations for ACOPF which attain very tight and accurate lower bounds. While our relaxations outer-approximate the Jabr and $i2$ rotated cones~\cite{jabr06,coffrin+etal16a}, empirical results show that the SDP relaxations for ACOPF~\cite{bai+etal08,lavaei+low12,kocuk+etal18}, can yield even tighter bounds, though at a frequently much higher computational cost. Nevertheless, clique decomposition techniques that leverage the sparsity of the network~\cite{fukuda+etal00,molzahn+etal13,andersen+etal20,anjos+etal21} show promise. However, these advances are still dependent on SDP solvers, which often fail to fully scale or guarantee valid bounds (invalid bounds are reported in~\cite{kocuk+etal18, oustry+etal22}). Moreover, strong numerical assumptions on the problem data, such as enforcing lower bounds on the minimum resistance for transmission lines, are typically made and seem to be necessary for solver convergence.




\subsection{Our contributions}

\begin{itemize}
\setlength\itemsep{0.5em}
    \item We propose a linear cutting-plane approach that offers fast and robust theoretical guarantees, providing bounds that closely match those of SDP relaxations; in contrast, relying on nonlinear (convex) relaxations to the SDP cone can yield invalid bounds as reported in e.g., \cite{kocuk+etal18, oustry+etal22}. 
    \item We propose a disciplined numerically stable and scalable approach to the SDP relaxations which exploits low-dimensional structures.
    \item As in~\cite{bienstock+villagra25}, our cutting-plane procedure is warm-startable and is capable of delivering practicable and tight relaxations for multi-period formulations for ACOPF.
    \item We provide new insights regarding two well-known SDP formulations for ACOPF, and provide theoretical justification for the tightness of these relaxations.
    \item Our development of SDP cuts aligns the goal of developing tight relaxations for ACOPF, which do not depend on power grid modeling features.
\end{itemize}

\section{Problem Formulation and SDP Relaxations}

\subsection{ACOPF}

 Let $\mathcal{B}$ denote the set of buses, $\mathcal{E}$ the set of branches and $\mathcal{G}$ the set of generators;  for each bus $k \in \mathcal{B}$, $\mathcal{G}_{k} \subseteq \mathcal{G}$ is the set of generators at $k$. Each bus $k$ has fixed active complex load $S_{k} = P_{k}^{d} + j Q_{k}^{d}$, and lower $\underline{V}_{k} \geq 0$ and upper $\overline{V}_{k} \geq 0$ voltage magnitude limits. For each branch $\{k,m\}$ we are given a thermal limit $0 \leq U_{km} \leq +\infty$. The goal is to find a vector of complex voltage phasors $V$, and active $P^{g}$ and reactive $Q^g$ power generation for every generator $g$, so that power is transmitted by the network to satisfy power demands at minimum cost. Using the so-called \textit{complex-voltage representation}~\cite{bose+etal11} we obtain the following quadratically constrained optimization problem in complex variables: 
\begin{subequations}\label{ACOPF-C}
\begin{align} 
&[\text{ACOPF-C}]: \hspace{2em} \min \hspace{2em} \sum_{k \in \mathcal{G}} F_{k}(P_{k}^{g}) \label{ACOPF-C:theobjective} \\
    &\text{subject } \text{to:} \nonumber\\
    &\forall \, k \in \mathcal{B}: \nonumber \\
    &\ \underline{V}_{k}^{2} \leq |V_{k}|^{2} \leq \overline{V}_{k}^{2}, \label{ACOPF-C:voltlimit} \\
    &\ v_{k}^{(2)} = |V_{k}|^{2} \label{ACOPF-C:def_volt2} \\
    &\sum_{\{k,m\} \in \delta(k)} P_{km} = \sum_{\ell \in \mathcal{G}_{k}} P_{\ell}^{g} - P_{k}^{d}, \label{ACOPF-C:activepowerbal}\\
    &\sum_{\{k,m\} \in \delta(k)} Q_{km} = \sum_{\ell \in \mathcal{G}_{k}} Q_{\ell}^{g} - Q_{k}^{d}, \label{ACOPF-C:reactivepowerbal}\\
    &\forall \{k,m\} \in \mathcal{E}: \nonumber \\
    &\ P_{km} = G_{kk}v_{k}^{(2)} + G_{km} c_{km} + B_{km} s_{km}, \label{ACOPF-C:Pdef_from}  \\
    &\ P_{mk} = G_{mm}v_{m}^{(2)} + G_{mk} c_{km} -  B_{mk} s_{km}, \label{ACOPF-C:Pdef_to} \\
    &\ Q_{km} = - B_{kk}v_{k}^{(2)} - B_{km} c_{km} + G_{km} s_{km}, \label{ACOPF-C:Qdef_from} \\
    &\ Q_{mk} = - B_{mm}v_{m}^{(2)} - B_{mk} c_{km} - G_{mk} s_{km}, \label{ACOPF-C:Qdef_to}\\
    &\ c_{km} = \frac{V_{k}V_{m}^{*} + V_{k}^{*}V_{m}}{2}, \quad s_{km} = \frac{V_{k}V_{m}^{*} - V_{k}^{*}V_{m}}{2j}, \label{ACOPF-C:def_cs}\\
    &\ \max\{ P_{km}^{2} + Q_{km}^{2}, P_{mk}^{2} + Q_{mk}^{2} \} \leq U_{km}^{2} \label{ACOPF-C:capacity} \\
    &\forall k \in \mathcal{G}: \underline{P}^{g}_{k} \leq P_{k}^g \leq \overline{P}^{g}_{k} \quad \underline{Q}^{g}_{k} \leq Q_{k}^g \leq \overline{Q}_{k}. \label{ACOPF-C:genpowerlimit}
\end{align}
\end{subequations}


Above, $G_{kk}, B_{kk}, G_{km}, B_{km}, G_{mk}, B_{mk}, G_{mm}$ and $B_{mm}$ are physical parameters of  branch $\{k,m\}$; \eqref{ACOPF-C:capacity} amounts to flow capacity constraints; \eqref{ACOPF-C:voltlimit} and~\eqref{ACOPF-C:genpowerlimit} impose operational limits; and~\eqref{ACOPF-C:activepowerbal}-\eqref{ACOPF-C:reactivepowerbal} impose active and reactive power balance at each bus $k$. For each $k \in \mathcal{G}$, the functions $F_{k} : \mathbb{R} \to \mathbb{R}$~\eqref{ACOPF-C:theobjective} are usually convex and piecewise-linear or quadratic. Constraints on voltage angle differences are often either not present or concern maximum angle differences that are \textit{small} (smaller than $\pi/2$)); under such conditions, they can be equivalently restated as thermal limit constraints~\eqref{ACOPF-C:capacity}. Hence, we do not explicitly include them in our formulations.


On the other hand, by writing the complex voltage phasors using rectangular coordinates, i.e., $V_{k} = e_{k} + j f_{k}$ for $k \in \mathcal{B}$, we have the following equivalent QCQP formulation in real variables for ACOPF~\cite{bai+etal08}:
\begin{subequations}\label{ACOPF-R}
\begin{align} 
&[\text{ACOPF-R}]: \hspace{2em} \min \hspace{2em} \sum_{k \in \mathcal{G}} F_{k}(P_{k}^{g}) \label{ACOPF-R:theobjective} \\
    &\text{subject } \text{to:} \quad \eqref{ACOPF-C:voltlimit},~\eqref{ACOPF-C:activepowerbal}-\eqref{ACOPF-C:Qdef_to},~\eqref{ACOPF-C:genpowerlimit}-\eqref{ACOPF-C:capacity} \nonumber\\
    &\forall \, k \in \mathcal{B}: \nonumber \\
    &\ v_{k}^{(2)} = e_{k}^{2} + f_{k}^{2}, \label{ACOPF-R:def_volt2} \\
    &\forall \{k,m\} \in \mathcal{E}: \\
    &\ c_{km} = e_{k}e_{m} + f_{k}f_{m}, \quad  s_{km} = e_{m}f_{k} - e_{k}f_{m}, \label{ACOPF-R:def_cskm}
\end{align}
\end{subequations}

Please refer to the survey~\cite{bienstock+etal22} for equivalent ACOPF formulations.

\subsection{SDP relaxations for ACOPF}

We consider three SDP relaxations of ACOPF problem. The first, 
which we denote by SDP-C~\cite{bose+etal11}, is a complex-valued SDP relaxation of~\eqref{ACOPF-C}. The second formulation, which we denote by SDP-R~\cite{shor87}, arises from applying the Shor relaxation~\cite{shor87} to the QCQP~\eqref{ACOPF-R} in rectangular coordinates, where only second-order terms are considered. Finally, SDP-CR is obtained by converting SDP-C into an equivalent real-valued formulation.

Before we introduce the relaxations, we fix some notation. We denote by $n := |\mathcal{B}|$, $(\R^{2n \times 2n},\langle \cdot, \cdot \rangle)$ and $(\C^{n \times n},\langle \cdot, \cdot \rangle)$ the Euclidean spaces of real matrices and complex matrices, respectively. We denote by $(\cdot)^{*}$ the Hermitian transpose. The standard (Frobenius) inner-product $\langle \cdot, \cdot \rangle$ is defined as $\langle A, B \rangle := \trace(A^{*} B)$ for matrices $A,B \in \Smat^{n}$ or $A,B \in \Hmat^{n}$. This inner-product induces the Euclidean norm $|| A || := \sqrt{\langle A, A \rangle}$. We use $A\succeq 0$ to denote that $A$ is positive semidefinite (PSD), interpreting this over $\R$ or $\C$ according to the nature of $A$.  

\begin{remark}
We note that the following semidefinite programming formulations are not presented in standard form. By means of an epigraph reformulation and the Schur complement, see e.g.~\cite{lavaei+low12,bienstock+etal22}, the objective can be expressed as a linear function in the matrix variable. Moreover, by applying the Schur complement again,~\eqref{ACOPF-C:capacity} can be reformulated as a PSD constraint.
\end{remark}

\paragraph{SDP-C}

Following ideas in~\cite{bose+etal11}, constraints~\eqref{ACOPF-C:def_cs} can be written as  $c_{km} = \langle R_{km}, VV^{*} \rangle$ and $s_{km} = \langle T_{mk}, VV^{*} \rangle$ with $R_{km}, T_{km} \in \C^{n \times n}$ Hermitian. Therefore, we can write $c_{km} = \langle R_{km}, X \rangle$ and $s_{km} = \langle T_{mk}, X \rangle$ with $X = VV^{*}$, i.e., $X$ has rank one. We can relax the latter and approximate the outer-product $VV^{*}$ by a PSD matrix $X \in \mathbb{C}^{n \times n}$, and obtain the following SDP relaxation in complex numbers:
\begin{subequations}\label{SDP-C}
\begin{align} 
&[\text{SDP-C}]: \hspace{2em} \min \hspace{2em} \sum_{k \in \mathcal{G}} F_{k}(P_{k}^{g}) \label{SDP-C:theobjective} \\
    &\text{subject } \text{to:} \quad \eqref{ACOPF-C:voltlimit},~\eqref{ACOPF-C:activepowerbal}-\eqref{ACOPF-C:Qdef_to},~\eqref{ACOPF-C:genpowerlimit}-\eqref{ACOPF-C:capacity} \nonumber\\
    &\forall \, k \in \mathcal{B}: \nonumber \\
    &\ v_{k}^{(2)} = X_{kk}, \label{SDP-C:def_volt2} \\
    &\forall \{k,m\} \in \mathcal{E}: \\
    &\ c_{km} = \frac{X_{km} + X_{mk}}{2}, \quad  s_{km} = \frac{X_{km} - X_{mk}}{2j}, \label{SDP-C:def_cskm} \\
    &\ X \succeq 0, \quad X \in \mathbb{C}^{n \times n} \label{SDP-C:psdconstraint}
\end{align}
\end{subequations}


\paragraph{SDP-R} On the other hand, we approximate the outer-product $(e,f)^\top (e,f)$ with a PSD matrix $W \in \mathbb{R}^{2 n \times 2n}$, i.e., we relax the rank constraint in~\eqref{ACOPF-R}, hence we obtain
\begin{subequations}\label{SDP-R:firstformulation}
\begin{align} 
&[\text{SDP-R}]: \hspace{2em} \min \hspace{2em} \sum_{k \in \mathcal{G}} F_{k}(P_{k}^{g}) \label{SDP-R:theobjective} \\
    &\text{subject } \text{to:}  \quad \eqref{ACOPF-C:voltlimit},~\eqref{ACOPF-C:activepowerbal}-\eqref{ACOPF-C:Qdef_to},~\eqref{ACOPF-C:genpowerlimit}-\eqref{ACOPF-C:capacity} \nonumber\\
    &\forall \, k \in \mathcal{B}: \nonumber \\
    &\ v_{k}^{(2)} = W_{kk} + W_{k'k'}, \label{SDP-R:def_volt2} \\
    &\forall \{k,m\} \in \mathcal{E}: \\
    &\ c_{km} = W_{km} + W_{k'm'}, \,  s_{km} = W_{mk'} - W_{km'}, \label{SDP-R:def_cskm} \\
    &\ W \succeq 0, \quad W \in \R^{2n \times 2n} \label{SDP-R:psdconstraint}
\end{align}
\end{subequations}

\paragraph{SDP-CR}

Next, we present SDP-CR. We first elaborate on a procedure, known as \emph{realification} of an Hermitian matrix, which reduces complex semidefinite programs to real ones. Let $\mathbb{H}^{n} \subseteq \mathbb{C}^{n \times n}$ be the subspace of $(n \times n)$-dimensional Hermitian matrices and $\mathbb{B}^{2n} \subseteq \mathbb{S}^{2n}$ the subspace of $(2n\times 2n)$-dimensional real symmetric matrices of the form
    \begin{equation*}
    \begin{pmatrix}
       A  & -B \\
       B  & A
    \end{pmatrix}
    \end{equation*}
where $A, B \in \mathbb{R}^{n \times n}$ with $A = A^{\top}$ and $B^{\top} = - B$. Let $L : \mathbb{H}^{n} \to \mathbb{B}^{2n}$ be the map defined by:
\begin{equation*}
    L(X) := 
    \begin{pmatrix}
        \Re X & - \Im X \\
        \Im X & \Re X
    \end{pmatrix} \quad \text{for $X \in \mathbb{H}^{n}$}.
\end{equation*}
It is straightforward to verify that this map is a linear bijection. The following facts are standard: a proof of Lemma~\ref{quadraticform_realvalued} can be found in any linear algebra text, 
while a proof of Lemma~\ref{realification_psd} can be found in~\cite{goemans+williamson04}.

\begin{lemma}\label{quadraticform_realvalued}
    Let $X \in \mathbb{C}^{n}$. Then the quadratic form $\mathbb{C}^{n} \to \mathbb{C}$ defined as $u \mapsto u^* X u$ is real-valued.
\end{lemma}

\begin{lemma}\label{realification_psd}
    Let $X \in \mathbb{H}^{n \times n}$. Then $X \succeq 0$ if and only if $L(X) \succeq 0$.
\end{lemma}

Therefore, SDP-CR is obtained by applying the realification to $X$ in SDP-C. Moreover, in~\cite{goemans+williamson04} it is shown that these two SDP relaxations are equivalent, in the sense that an optimal solution to SDP-C can be obtained from an optimal solution to SDP-CR, and vice versa.

\subsection{Brief review on prior work on SDP relaxations}

There is an extensive literature on SDP relaxations for the ACOPF problem. In~\cite{bai+etal08}, the real SDP relaxation (SDP-R) was introduced and later further developed in~\cite{lavaei+low12}. Conversely, the complex SDP relaxation (SDP-C) was proposed in~\cite{bose+etal11} and subsequently used in~\cite{bose+etal12} to establish the equivalence between the branch-flow and bus-injection models of ACOPF.
Chordal sparsity techniques were later explored in~\cite{molzahn+etal13,andersen+etal14,andersen+etal20,anjos+etal21}. Finally, moment relaxations have been studied in~\cite{molzahn+hiskens15,josz+etal15,josz+molzahn18}. In all of these works, a small minimum resistance is enforced on all transmission lines, which appears to be necessary for achieving numerical convergence of the corresponding SDP implementations. In contrast, our algorithms make no assumptions on the data.

We also note that in~\cite{kocuk+etal18}, cuts are introduced to strengthen an SOC relaxation of ACOPF, although their computation requires solving an SDP. Other types of cuts have also been proposed, such as the polynomial determinant cuts in~\cite{hijazi+etal16}, which, however, are nonlinear polynomial inequalities.

\section{On the relationship between SDP-R and SDP-CR}\label{section:SDP-C_SDP-R}

In this section, we analyze the properties of the SDP-R relaxation, identify statements in the literature concerning its relationship with SDP-CR that are incorrect, and point out assumptions that should not have been made.

A widespread belief is that SDP-C and SDP-R are equivalent. To the best of our knowledge, nearly all papers on SDP relaxations for the ACOPF problem~\cite{bose+etal11,bose+etal12,lavaei+low12,molzahn+etal13, andersen+etal14, molzahn+hiskens15, josz+etal15, hijazi+etal16, coffrin+etal16a,josz+molzahn18,  andersen+etal20, anjos+etal21} study one of these formulations, typically SDP-C or its equivalent form SDP-CR, without comparing them, implicitly treating them as reformulations of one another. Some works even refer to either formulation as \emph{the} SDP relaxation for ACOPF.


In~\cite{taylor15}, an effort was made to unify various SDP relaxations within a common framework. In particular, \cite{taylor15} claims that the matrix variable $W$ in SDP-R is equivalent, up to an orthogonal similarity transformation, to the matrix variable $L(X)$ in SDP-CR, i.e., that there exists some orthogonal $Q \in \mathbb{R}^{2n \times 2n}$ such that $W = Q^\top L(X) Q$. This statement has subsequently been cited in works such as~\cite{kocuk+etal16,kocuk+etal18}. Unfortunately, this statement is incorrect (see Appendix~\ref{appendix:W_and_L(X)} for a proof). This commonly held yet unproven belief has played a central role in the literature on SDP relaxations, and, more broadly, on convex relaxations of the ACOPF problem.

This issue is significant because relaxations are typically evaluated relative to one another. For example, the perceived superiority of \emph{the} SDP relaxation over the Jabr SOC relaxation~\cite{jabr06} rests on the fact that the positive semidefiniteness of the matrix variable in SDP-C implies the Jabr inequality. However, to the best of our knowledge, no proof has been established showing that SDP-R also implies this inequality. This distinction affects all works comparing convex SDP formulations for ACOPF, particularly those that rely on the presumed strength of SDP-R under the (unproven) assumption that it enforces fundamental inequalities such as the Jabr inequality.

However, we have a positive result: below we show that SDP-R is at least as strong as SDP-C. We do this by exhibiting a feasible solution to SDP-C from a feasible solution to SDP-R, both with same objective value.

\begin{definition}\label{definition:SDP-R_to_SDP-C}
    Let $W \in \mathbb{R}^{2n \times 2n}$ be an arbitrary symmetric matrix. We define a Hermitian matrix $X^{W} \in \C^{n \times n}$ associated to $W$ as follows: 
    \begin{align*}
        X^{W}_{kk} :&= W_{kk} + W_{k'k'} \\
        \Re X^{W}_{km} :&= W_{km} + W_{k'm'} \\
        \Im X^{W}_{km} :&= W_{km'} - W_{mk'}.
    \end{align*}
\end{definition}

\begin{theorem}\label{theorem:W_PSD_X^W_PSD}
    Let $W \in \mathbb{R}^{2n \times 2n}$ be an arbitrary symmetric matrix. Then $W \succeq 0$ implies $X^{W} \succeq 0$.
\end{theorem}
\begin{proof}
    Since $W$ is PSD over the reals, the following bilinear map is a semi-inner product in $\mathbb{C}^{n}$: 
    $$\langle u, v \rangle_{W} := u^{*} W v, \hspace{2em} \forall u,v \in \C^{n}.$$ 
    This follows because $\langle \cdot, \cdot \rangle_{W}$ satisfies: (i) \emph{linearity} in the second argument; (ii) \emph{conjugate symmetry}, i.e., $\langle u, v \rangle_{W} = \langle v, u \rangle_{W}^{*}$; and (iii) \emph{semipositivity}, i.e., it holds that $\langle u, u \rangle_{W} \geq 0$ for every $u \in \mathbb{C}^n$, by Lemma~\ref{realification_psd} since $L(W)$ is a block-diagonal PSD matrix.

    Consider the following points in $\C^{n}$: $z_{k} := \delta_{k} + j \delta_{k'}$ for $k \in \mathcal{B}$, where $\delta_{k}$ denotes the $k$th canonical basis vector in $\R^{2n}$. We note that
    \begin{align*}
        \langle z_{k}, z_{k} \rangle_{W} &= \delta_{k}^{*} W \delta_{k} + \delta_{k'}^{*} W \delta_{k'} = W_{kk} + W_{k'k'}  \\
        \langle z_{k}, z_{m} \rangle_{W} &= \delta_{k}^{*} W \delta_{m} + j \delta_{k}^{*} W \delta_{m'} -  j \delta_{k'}^{*} W \delta_{m} + \delta_{k'}^{*} W \delta_{m'} \\
        &= (W_{km} + W_{k'm'}) + j (W_{km'} - W_{mk'}) \\
        \langle z_{m}, z_{k} \rangle_{W} &= \delta_{m}^{*} W \delta_{k} + j \delta_{m}^{*} W \delta_{k'} -  j \delta_{m'}^{*} W \delta_{k} + \delta_{m'}^{*} W \delta_{k'} \\
        &= (W_{km} + W_{k'm'}) + j (W_{mk'} - W_{km'})
        \end{align*}
    where we used the fact that $W$ is symmetric. Next, we construct the following ``semi'' Gram matrix in $\mathbb{C}^{n \times n}$: for $k,m \in \mathcal{B}$,
    \begin{align*}
        G_{km} :&=  \left\{
        \begin{array}{cc}
        \langle z_{k}, z_{k} \rangle_{W}     & \mbox{ if $k = m$}, \\
        \langle z_{k}, z_{m} \rangle_{W}     & \mbox{ if $k \neq m$}.
        \end{array}
        \right.
   \end{align*}
   Since $\langle z_{m}, z_{k} \rangle_{W} = \langle z_{k}, z_{m} \rangle_{W}^{*}$, we deduce that $G$ is Hermitian. We claim that $G$ is PSD over $\mathbb{C}^n$. Indeed, let $y \in \mathbb{C}^{n}$ be arbitrary, then
   \begin{align*}
       y^{*} G y &= \sum_{k=1}^{n} \sum_{m=1}^{n} y_{k}^{*} y_{m} \langle z_{k},z_{m} \rangle_{W} \\
       &= \sum_{k=1}^{n} \sum_{m=1}^{n} \langle y_{k} z_{k}, y_{m} z_{m} \rangle_{W} \\
       &= \sum_{k=1}^{n} \langle y_{k} z_{k}, \sum_{m=1}^{n} y_{m} z_{m} \rangle_{W} \\
       &= \langle \sum_{k=1}^{n} y_{k} z_{k} , \sum_{k=1}^{n} y_{k} z_{k}  \rangle_{W} \\
       &\geq 0
   \end{align*}
   where the second equality follows since $\langle \cdot, \cdot \rangle_{W}$ satisfies conjugate symmetry and it is linear in the second argument, and nonnegativity follows by the semipositivity of $\langle \cdot, \cdot \rangle_{W}$. Thus, given that $G = X^{W}$, we conclude that $X^{W}$ is PSD.
\end{proof}

Moreover, from our results above we can deduce that the Jabr inequality is implied by the set of feasible solutions to SDP-W. In particular, it is implied by the constraint $W \succeq 0$. Notably, we uncover what we term \emph{permuted Jabr inequalities}, by exploiting the permutation symmetries of~\eqref{jabr_W}. This provides a theoretical justification for the tightness of the SDP-R relaxation since the \emph{loss inequalities} are implied by the Jabr inequality; see Section 4.3~\cite{bienstock+villagra25} for a discussion.

Let $A \in \mathbb{K}^{r \times r}$ where $K \in \{ \R, \C\}$ and $r \in \mathbb{N}$. We denote by $A_{R}$ the principal submatrix of $A$ given by a set $R \subseteq [r]$ of row and column indices.

\begin{corollary}
Consider any $2 \times 2$ principal submatrix of $X^{W}$
\begin{equation*}
X^{W}_{\{k,m\}} := 
\left(
    \begin{array}{cc}
        X^{W}_{kk} & \Re X^{W}_{km} + j \Im X^{W}_{km}  \\
        \Re X^{W}_{mk} + j \Im X^{W}_{km} & X^{W}_{mm}
    \end{array}
\right).
\end{equation*}
If $W \succeq 0$, then $X^{W}_{\{k,m\}} \succeq 0$, i.e., $X_{kk}^{W}, X_{mm}^{W} \geq 0$ and
\begin{equation}\label{jabr_X}
    (\Re X_{km}^{W})^{2} + (\Im X_{km}^{W})^{2} \leq X_{kk}^{W} X_{mm}^{W}.
\end{equation}
In particular, the following inequality holds
\begin{equation}\label{jabr_W}
\scalebox{0.72}{$
    (W_{km} + W_{k'm'})^2 + (W_{mk'} - W_{km'})^2 \leq (W_{kk} + W_{k'k'})(W_{mm} + W_{m'm'}).
$}
\end{equation}
Moreover, for any row-column permutation $\sigma$ of the submatrix $W_{\{k,m,k',m'\}}$, the following inequality also holds:
\begin{align}\label{permuted_jabr_W}
  c_{\sigma}^{2} + s_{\sigma}^{2} \leq r_{\sigma}t_{\sigma}.
\end{align}
where $c_{\sigma} := W_{\sigma(k)\sigma(m)} + W_{\sigma(k')\sigma(m')}$, $s_{\sigma} := W_{\sigma(k)\sigma(m')} - W_{\sigma(m)\sigma(k')} $, $r_{\sigma} := W_{\sigma(k)\sigma(k)} + W_{\sigma(k')\sigma(k')}$, $t_{\sigma}:= W_{\sigma(m)\sigma(m)} + W_{\sigma(m')\sigma(m')}$. 
\end{corollary}
\begin{proof}
   The first claim follows by a direct application of Theorem~\ref{theorem:W_PSD_X^W_PSD} to the principal submatrix $W_{\{k,m,k',m'\}}$ of $W$. The inequality~\ref{jabr_W} follows immediately by~\eqref{jabr_X} and the definition of $X^{W}$. Finally, inequality~\eqref{permuted_jabr_W} holds for any permutation $\sigma$ on $\{k,m,k',m'\}$ since any row-column permutation of $W_{\{k,m,k',m'\}}$ can be represented by a similarity transformation $Q^{T} W_{\{k,m,k',m'\}} Q$ where $Q$ is a permutation matrix (in particular orthogonal). Since similarity transformations are PSD-preserving we are done.
\end{proof}

Below we state and prove the main result of this section.

\begin{theorem}\label{theorem:SDP-R_SDP-C}
    Let $z_{C}$ and $z_{R}$ denote the optimal values of SDP-C and SDP-R, respectively. Then $z_{C} \leq z_{R}$.
\end{theorem}
\begin{proof}
    If SDP-R is infeasible, then the claim is vacuously true since $z_{R} = + \infty$. Next, suppose that SDP-R is feasible and let $y_{R}:= (\hat{P}^{g},\hat{Q}^{g},\hat{W},\hat{P},\hat{Q}, \hat{c},\hat{s}, \hat{v}^{(2)})$ be a feasible solution. We show how to obtain a feasible solution to SDP-C from it. Indeed, using $\hat{W}$ we can construct by Definition~\ref{definition:SDP-R_to_SDP-C} the Hermitian matrix $\hat{X}^{\hat{W}} \in \C^{n \times n}$. We claim that $y_{C}:= (\hat{P}^{g},\hat{Q}^{g},\hat{X}^{\hat{W}},\hat{P},\hat{Q}, \hat{c},\hat{s},\hat{v}^{(2)})$ satisfies all the constraints of SDP-C. Indeed, clearly $\hat{P}^{g},\hat{Q}^{g},\hat{P},\hat{Q}, \hat{c},\hat{s}, \hat{v}^{(2)}$ satisfy constraints~\eqref{ACOPF-C:voltlimit},~\eqref{ACOPF-C:activepowerbal}-\eqref{ACOPF-C:Qdef_to},~\eqref{ACOPF-C:genpowerlimit}-\eqref{ACOPF-C:capacity}. Moreover, since $\hat{X}^{\hat{W}}_{kk} = \hat{W}_{kk} + \hat{W}_{k'k'} = v_{k}^{(2)}$, then constraint~\eqref{SDP-C:def_volt2} holds. Next, we note that constraints~\eqref{SDP-C:def_cskm} are equivalent to $c_{km} = \Re X_{km}$, $s_{km} = \Im X_{km}$. Hence, since $\hat{c}_{km} = \hat{W}_{km} + \hat{W}_{k'm'}$ and $s_{km} = \hat{W}_{km'} - \hat{W}_{mk'}$, we have that $\hat{c}_{km} = \Re \hat{X}^{\hat{W}}_{km}$ and $\hat{s}_{km} = \Im \hat{X}^{\hat{W}}_{km}$, i.e.,~\eqref{SDP-C:def_cskm} is satisfied. Finally, given that $\hat{W} \succeq 0$, Theorem~\ref{theorem:W_PSD_X^W_PSD} implies that $\hat{X}^{\hat{W}} \succeq 0$ and~\eqref{SDP-C:psdconstraint} holds. Since $y_{C}$ and $y_{R}$ have the same objective value, we conclude that $z_{C} \leq z_{R}$.
\end{proof}

\section{Our disciplined approach to SDPs}

An SDP relaxation with matrix variable $X \succeq 0$  can be reformulated using a clique decomposition technique applied to the underlying power system network~\cite{fukuda+etal00}. This approach results in an equivalent SDP problem but with PSD constraints $X_{i} \succeq 0$ on much smaller matrices than $X$ and with linear constraints. The size of these matrices is relatively small (e.g., for case ACTIVSg70k with $70,000$ nodes, a clique decomposition may find a clique of size $114$). In~\cite{molzahn+etal13,andersen+etal14,anjos+etal21} different clique merging techniques and reformulations have been studied. 

Our approach can be outlined as follows. Given some decomposition $\{ X_{i} \}_{i \in \mathcal{D}}$, for each matrix $X_{i}$, we derive sparse, numerically valid, and stable linear inequalities which outer-approximate the PSD cone $X_{i} \succeq 0$. To maintain speed and numerical robustness we apply modern cut-management techniques, as in~\cite{bienstock+villagra25}, to several cut families, including eigenvector-based cuts~\cite{qualizza+etal12}. Our goal is to outer-approximate the feasible region of SDP-C. As in prior works in sparse SDPs, the philosophy is to keep the formulation as \emph{light} as possible, without adding too many \emph{dense} cuts. Currently, we have experiments with two decomposition heuristics: (i) identifying all the $3$-cliques in the network so that we do not have to add extra variables; and (ii) using a chordal decomposition~\cite{fukuda+etal00} of the network $(\mathcal{B},\mathcal{E})$.

In what follows, we also study the underlying geometry and related algebraic issues, with the aim of delivering \textit{fully valid} linear inequalities.

\subsection{SDP Cuts}\label{sdpcuts}

We denote by $\Hmat_{+}^{n}$ the PSD cone in $\C^{n \times n}$. Let $X_{0} \in \C^{n}$ be Hermitian. Suppose that $X_{0}$ is not positive semidefinite, and consider an eigen-decomposition (normalized) of $X_{0}$, i.e., $X_{0} = Q \Lambda Q^{*}$ where $\{q_{1}, \dots, q_{n}\} \subseteq \C^{n}$ are eigenvectors of $X_{0}$ which correspond to the columns of $Q$, and $\Lambda$ is a diagonal matrix whose elements along the diagonal are the eigenvalues of $X_{0}$ sorted in non-increasing order $\lambda_{1} \geq \cdots \geq \lambda_{n}$. 

Suppose that we want to separate $X_{0}$ from the PSD cone $\Hmat_{+}^{n}$. We discuss two types of cuts. For further discussion of (sparse) cutting planes for QCQPs, see~\cite{qualizza+etal12}. 

\paragraph{Eigenvector-based cuts}

Consider the eigenvector $q_{n} \in \C^{n}$ associated with the eigenvalue $\lambda_{n}$, i.e. $X_{0} q = \lambda_{n} q$ and $||q||_{2} = 1$. Since we assumed that $X_{0}$ is indefinite, we have $\lambda_{n} < 0$. Recall that $X \in \Hmat^{n}_{+}$ if and only if $x^{*} X x \geq 0$ for all $x \in \C^{n}$. Let $A := qq^{*}$. Hence, the following inequality defines a valid linear cut, which we refer to as an \emph{eigen-cut} for $\Hmat_{+}^{n}$,
\begin{equation}\label{eigen_cut}
    \langle A, X \rangle = q^{*} X q \geq  0, \hspace{1em} \forall X \in \Hmat^{n}
\end{equation}
Clearly, $\langle A, X_{0} \rangle = q^\top X_{0} q = \lambda_{n} < 0$. In other words, $X_{0}$ violates this inequality by $-\lambda_{n}$ and its distance to~\eqref{eigen_cut} is $|\lambda_{n}|$.

\begin{remark}
    It is not hard to show that eigen-cuts derived from $2 \times 2$ principal submatrices of $X_{0}$ correspond to the standard Jabr outer-approximation cuts in~\cite{bienstock+villagra25}.
\end{remark}

\paragraph{Projection-based cuts}\label{projection_based_cut}

Next, we derive a cut using the Euclidean projection of $X_{0}$ onto $\Hmat_{+}^{n}$. This is also known as the \emph{maximum-distance cut}. Consider the problem of minimizing $|| X - X_{0} ||_{F}^{2}$ subject to $X \in \Hmat_{+}^{n}$. It is known that the unique minimizer to this SDP is given by $X^{*} := \sum_{\lambda_{i} > 0} \lambda_{i} q_{i} q_{i}^{*}$, see e.g.~\cite{goulart+etal20}. Using a well-known characterization of the projection operator, we can derive the following valid linear inequality for $\Hmat_{+}^{n}$,
\begin{equation}\label{proj_inequality}
    \langle X_{0} - X^{*}, X - X^{*} \rangle \leq 0, \hspace{2em} X \in \Hmat_{+}^{n}.
\end{equation}
The left-hand side of the inequality is real, as both $X_{0} - X^{*}$ and $X - X^{*}$ are Hermitian. Thus, the inequality is well defined. Let $A := X^{*} - X_{0} = \sum_{\lambda_{i} < 0} (-\lambda_{i}) q_{i} q_{i}^{*}$. Then,~\eqref{proj_inequality} can be rewritten as $\langle X^{*} - X_{0}, X \rangle \geq \langle X^{*} - X_{0}, X^{*} \rangle$, i.e., 
\begin{equation}
    \langle A, X \rangle \geq 0  \label{projection_cut}
\end{equation}
since $\langle X^{*} - X_{0}, X^{*} \rangle = 0$. The matrix $X_{0}$ violates this inequality by $\sum_{\lambda_{i} < 0} \lambda_{i}^{2}$ and its distance to~\eqref{projection_cut} is $\sqrt{\sum_{\lambda_{i} < 0} \lambda_{i}^{2}}$.

\paragraph{Cut computations}

Let $A, X \in \Hmat^{n}$. Recall that $\langle A, X \rangle = \sum_{k,m} A_{km}^{*} X_{km}$ is real and consider the summands $A_{km}^{*} X_{km}$ and $A_{mk}^{*} X_{mk}$. We note that
\begin{align*}
    A_{km}^{*} X_{km} &= (\Re A_{km} - j \Im A_{km})(\Re X_{km} + j \Im X_{km}) \\
                      &= \Re A_{km} \Re X_{km} + \Im A_{km} \Im X_{km} \\
                      &+ j (\Re A_{km} \Im X_{km} - \Re X_{km} \Im A_{km}) \\
    A_{mk}^{*} X_{mk} &= (\Re A_{mk} - j \Im A_{mk})(\Re X_{mk} + j \Im X_{mk}) \\
                      &= \Re A_{mk} \Re X_{mk} + \Im A_{mk} \Im X_{mk} \\
                      &+ j (\Re A_{mk} \Im X_{mk} - \Re X_{mk} \Im A_{mk})   
\end{align*}
Since $\Re X_{km} = \Re X_{mk}$, $\Re A_{km} = \Re A_{mk}$, $\Im X_{km} = - \Im X_{mk}$, $\Im A_{km} = - \Im A_{mk}$ we have that $A_{km}^{*} X_{km} + A_{mk}^{*} X_{mk}$ equals $2 ( \Re A_{km} \Re X_{km} + \Im A_{km} \Im X_{km})$. Therefore, $\langle A, X \rangle \geq 0$ can be written as 
\begin{equation}\label{sdp_cuts}
\scalebox{0.9}{$
    \sum_{k=1}^{n} A_{kk} v_{k}^{(2)} + \sum\limits_{\substack{1\le k<m\le n}} 2 ( \Re A_{km} c_{km} + \Im A_{km} s_{km}) \geq 0.
$}
\end{equation}
using definitions~\eqref{SDP-C:def_volt2} and~\eqref{SDP-C:def_cskm}. We conclude this section with a result on the numerical robustness of our methods.
\begin{proposition}
    Let $X_{0} \in \C^{3 \times 3}$ be a complex Hermitian matrix that approximates the outer product of three complex voltage phasors $VV^{*}$, where $V \in \C^{3}$ (hence the entries of $X_{0}$ are well-behaved). Let $\epsilon > 0$, and suppose that the smallest eigenvalue of $X_{0}$, which we denote as $\lambda_{3}$, satisfies $\lambda_{3} > - \epsilon$. Then the distance of $X_{0}$ to $\Hmat_{+}^{3}$ is at most $\sqrt{2} \epsilon$.
\end{proposition}
\begin{proof}
    Let $\lambda_{1} \geq \lambda_{2} \geq \lambda_{3}$ be the eigenvalues of $X_{0}$. Since the trace of $X_{0}$ is always strictly positive, we have that $X_{0}$ has at most two negative eigenvalues. Suppose that $\lambda_{2}, \lambda_{3} < 0$. Then by assumption $- \lambda_{2}, - \lambda_{3} < \epsilon$. Therefore, by our previous discussion we have that
    the distance of $X_{0}$ to the $\Hmat_{+}^{3}$ is at most $\sqrt{ \epsilon^{2} + \epsilon^{2}} = \sqrt{2} \epsilon$.
\end{proof}

\subsection{The Algorithm}

First, we define the linearly constrained model $M_{0}$ as model~\eqref{SDP-C} with only linear constraints and no explicit dependence on $X$, i.e., without~\eqref{SDP-C:def_volt2},~\eqref{SDP-C:def_cskm}, and~\eqref{SDP-C:psdconstraint}. We denote by $M$ our dynamic relaxation at some iteration of our algorithm, see Algorithm~\ref{thealgorithm}. We set $\mathcal{D}$ to be the set of all the $3$-cycles of the network $(\mathcal{B},\mathcal{E})$. Hence, $X_{i}$ is a $3\times 3$ Hermitian matrix for each $i \in \mathcal{D}$. 

Following~\cite{bienstock+villagra25}, in every round of our procedure, linear constraints will be added to and removed from $M$. Given a solution $\overline{y}$ to $M$, line~\eqref{violated_ineq} of~\eqref{thealgorithm} refers to the application of cut management techniques to separate $\overline{y}$ from the Jabr, i2 and Limit inequalities, see~\cite{bienstock+villagra25}. Next, we check whether the matrices $\{ X_{i}(\overline{y}) \}_{i \in D}$ are PSD by calling NumPy's SVD implementation. If $X_{i}(\overline{y})$ is not PSD, we separate it from the PSD cone $X_{i} \succeq 0$ using the linear inequalities~\eqref{sdp_cuts}. We use maximum-distance cuts with at most two negative eigenvalues. If the round $r$ surpasses the threshold $r^{*}$, we augment $\mathcal{D}$ by computing a chordal decomposition~\cite{fukuda+etal00} of the network. For further details on cut management please refer to~\cite{bienstock+villagra25}.

Finally, other input parameters for our procedure are: a time limit $T>0$; a number of admissible iterations without sufficient objective improvement $T_{ftol} \in \mathbb{N}$; and a threshold for relative objective improvement $\epsilon_{ftol} > 0$.

\begin{algorithm}
\caption{Cutting-Plane Algorithm}\label{thealgorithm}
\begin{algorithmic}[1]
\Procedure {Cutplane}{}
\State Initialize $r \gets 0$, $M \gets M_{0}$, $z_{0} \gets + \infty$
\State $\mathcal{D} \gets $ all the $3$-cycles of the network $(\mathcal{B},\mathcal{E})$
\While{$t < T$ and $r < T_{ftol}$}
\State $z \gets \min M$ and $\bar{y} \gets \text{argmin} \, M$
\State Apply cut management techniques in~\cite{bienstock+villagra25}\label{violated_ineq}
\State Check whether matrices $\{ X_{i}(\overline{y}) \}_{i \in D}$ are PSD
\State Compute cuts~\eqref{sdp_cuts} for $X_{i}(\overline{y})$ nonPSD
\State Drop cuts of age $\geq T_{age}$ whose slack is $\geq$ $\epsilon_{j}$
\State \textbf{If } $r \geq  r^{*}$ \textbf{ do } \textsc{CycleHierarchy}
\State \textbf{If } $z - z_{0} < z_{0} \cdot \epsilon_{ftol}$ \textbf{ do } $r \gets r+1$
\State \textbf{Else do} $r \gets 0$
\State $z_{0} \gets z$
\EndWhile
\EndProcedure
\end{algorithmic}
\end{algorithm}



\section{Preliminary computational results}


We ran our experiments on a Apple M2 Pro machine with $16$GB RAM. We used Gurobi 11.0.0 with its default homogeneous self-dual embedding interior-point algorithm (without \emph{Crossover}). We set the parameters \emph{Bar Homogeneous} and \emph{Numeric Focus} equal to $1$. Barrier convergence tolerance and absolute feasibility and optimality tolerances were set to $10^{-6}$. The Jabr SOC bounds in Table~\ref{table:loads_perturbed}, and the AC primal bounds in Tables~\ref{table:newbounds}-\ref{table:loads_perturbed} are reported in~\cite{bienstock+villagra25}.

\paragraph{Tighter lower bounds} In Table~\ref{table:newbounds} we report on bounds obtained by Algorithm~\ref{thealgorithm} with a time limit of $1200$ seconds as in Section 7.1.1~\cite{bienstock+villagra25}. After $r^{*} = 5$ rounds, the set $\mathcal{D}$ is augmented by cliques of length at most $5$. The column \#Cliques reports the number of 3-, 4-, and 5-cliques, respectively, handled in formulation $M$ at the final iteration. The column `EigRatio' shows the smallest ratio, among the matrices $X_{i}(\overline{y})$, between the largest and the second-largest eigenvalues -- a commonly used proxy for assessing rank-one behavior. Column `DInfs' reports on the maximum dual infeasibility error for the last solve, see Section 6~\cite{bienstock+villagra25}.

\paragraph{Loads perturbed} The load of each bus was perturbed by a Gaussian $(\mu,\sigma) = (0.01 \cdot P_{d}, 0.01 \cdot P_{d})$, where $P_{d}$ denotes the original load, subject to the newly perturbed load being non negative~\cite{bienstock+villagra25}. We warm-start Algorithm~\ref{thealgorithm} with precomputed cuts (Jabr, i2, and limit) in~\cite{bienstock+villagra25} with a time limit of $600$ seconds. Our SDP cuts are only for $3$-cycles in the network. Column `DInfs' reports on the maximum dual infeasibility error for the last solve, while column `Added' reports on the number of cuts in the formulation at the last solve.

\begingroup
\setlength{\tabcolsep}{3pt} 
\begin{table*}
\caption{New Lower Bounds with Linearly-Constrained Relaxations}
\centering
\begin{tabular}{ @{} l r r r r r r r r r r @{} }
\toprule
& \multicolumn{4}{c}{Advanced Cutting-Plane} & \multicolumn{1}{c}{Cutting-plane~\cite{bienstock+villagra25}} & \\
\cmidrule(l{0.5em}r{0.35em}){2-5} 
\multicolumn{1}{l}{Case} & Objective & \#Cliques & DInf & EigRatio & Objective & AC & \\
\midrule
14 & 8080.96 & (5,0,0) & 2.36e-11 & 682972.8 & 8074.70 & 8081.18 &  \\
1354pegase & 74013.67 & (87,111,44) & 3.31e-09 & 1.11 & 74011.00 & 74069.35 &  \\
9241pegase & 310226.80 & (8673,1029,484) & 4.30e-07 & 1.46 & 309221.81 & 315911.56 &  \\
13659pegase & 380055.79 & (8673,0,0) & 3.70e-05 & 1.83 & 379084.55 & 386108.81 &  \\
\bottomrule
\end{tabular}
\label{table:newbounds}
\end{table*}
\endgroup

\begingroup
\setlength{\tabcolsep}{3pt} 
\begin{table*}
\caption{Warm-Started Relaxations, Loads perturbed Gaussian $(\mu,\sigma) = (0.01 \cdot P_{d}, 0.01 \cdot P_{d})$}
\centering
\begin{tabular}{ @{} l r r r r r r r r r r r r r r r@{} }
\toprule
& \multicolumn{4}{c}{Advanced Cutting-Plane} & \multicolumn{4}{c}{Cutting-plane~\cite{bienstock+villagra25}} & \\
\cmidrule(l{0.5em}r{0.35em}){2-5} \cmidrule(l{0.5em}r{0.40em}){6-9}
\multicolumn{1}{l}{Case} & Objective & Time (s) & DInf & Added & Objective & Time (s) & DInfs & Added & Jabr SOC & AC & \\
\midrule
9241pegase & 310225.93 & 307.67 & 8.01e-10 & 73940 & 309288.30 & 15.01 & 9.52e-07 & 29875 & - & 315979.53 &  \\
ACTIVSg10k & 2475135.37 & 7.43 & 5.06e-07 & 22605 & 2475040.71 & 4.94 & 2.55e-08 & 18183 & - & 2484093.15 &  \\
13659pegase & 380518.18 & 429.05 & 5.46e-08 & 73889 & 379742.50 & 31.09 & 3.22e-06 & 37297 & - & 386765.25 &  \\
ACTIVSg25k & 5989222.32 & 32.36 & 7.21e-08 & 60690 & 5988885.07 & 20.03 & 1.34e-08 & 43851 & 5949381.04 & 6013477.05 &  \\
78484epigrids-api & 15866507.73 & 118.11 & 1.27e-05 & 315589 & 15862318.35 & 82.28 & 4.04e-07 & 240576 & 15859950.52 & - \\
78484epigrids-sad & 15175077.19 & 124.12 & 5.33e-06 & 397189 & 15176865.38 & 96.74 & 3.86e-06 & 313587 & 15174716.43 & 15316872.94 \\
\bottomrule
\end{tabular}
\label{table:loads_perturbed}
\end{table*}
\endgroup

\section{Future work}


We leave for future work further experimentation with SDP cut management techniques, and extensive benchmarking to public libraries. Furthermore, based on~\cite{molzahn+hiskens14,josz+molzahn18}, we are currently investigating how careful selection of low-order terms has the potential to strengthen our bounds.

\appendix

\begin{proposition}
    Let $X \in \C^{n \times n}$ and $W \in \R^{2n \times 2n}$ be arbitrary rank-one matrices. Then, there does not exist an orthogonal matrix $Q \in \R^{2n \times 2n}$ such that $W = Q L(X) Q^{\top}$.
\end{proposition}
\begin{proof}
    Assume, to the contrary, that there exists some orthogonal matrix $Q \in \R^{2n \times 2n}$ such that $W = Q L(X) Q^{\top}$. Since orthogonal similarity transformations preserve rank, and Lemma~\ref{appendix:W_and_L(X)} implies that $L(X)$ is rank-two, we have that $\text{rank } W = 2$, which is a contradiction. 
\end{proof}

\begin{lemma}\label{appendix:W_and_L(X)}
    Let $X \in \mathbb{C}^{n \times n}$. Then $\text{rank } L(X) = 2 \cdot \text{rank } X$.
\end{lemma}
\begin{proof}
    Let $z = x + j y \in \mathbb{C}^{n}$ where $x, y \in \mathbb{R}^{n}$. Consider the $\mathbb{R}$-linear map $\phi: \mathbb{C}^{n} \to \mathbb{R}^{2n}$ given by $\phi(z) = (x,y)^\top$. We note that
    \begin{align*}
        X z = 0 &\iff (\Re X + j \Im X) (x + j y) = 0 \\
        &\iff (\Re X x - \Im X y) + j (\Im X x + \Re X y) = 0 \\
        &\iff 
        \begin{pmatrix}
            \Re X & - \Im X \\
            \Im X & \Re X
        \end{pmatrix}
        \begin{pmatrix}
            x \\
            y
        \end{pmatrix}
        = 0 \\
        &\iff L(X)
        \begin{pmatrix}
            x \\
            y
        \end{pmatrix}
        = 0.
    \end{align*}
    Next, let $B$ be a basis for $\ker V$. By the above and since $\phi$ is injective, we have that $\dim \ker L(X) \geq |B|$. We show that $\dim \ker L(X) = 2 \cdot |B|$. For an arbitrary $z = x + j y \in \mathbb{C}^{n}$, we remark that clearly $z$ and $j z$ $\mathbb{C}$-linearly dependent, while $\phi(z)$ and $\phi(j z)$ are $\mathbb{R}$-linearly independent. Indeed, $\phi(jz) = (-y,x)^\top$ hence $\phi(z)$ is orthogonal to $\phi(jz)$. Since for different $z,z' \in B$, we have that $z$ and $jz'$ are $\mathbb{C}$-linearly independent, hence  $\dim \ker L(X) \geq 2 \cdot |B|$. Finally, this inequality cannot be strict since if it were, there would be some $(\hat{x},\hat{y}) \in \ker L(X)$ such that $\phi^{-1}(\hat{x},\hat{y}) \in \ker V$ is not spanned by $B$.
\end{proof}

\section*{Acknowledgment}
This work was supported by the Air Force Office of Scientific Research (AFOSR) under Grant No.~FA9550-23-1-0697.


\bibliographystyle{IEEEtran}
\bibliography{pscc2026.bib}
\end{document}